\documentclass[11pt,a4paper,reqno]{amsart}
\usepackage{amsthm,amsmath,amsfonts,amssymb,amsxtra,appendix,bookmark,dsfont,latexsym,bm,hyperref,delarray,color,euscript,amsgen,amsbsy,amsopn,amscd,latexsym,mathrsfs}

\makeatletter

\usepackage{hyperref}

\makeatletter
\def\@tocline#1#2#3#4#5#6#7{\relax
  \ifnum #1>\c@tocdepth 
  \else
    \par \addpenalty\@secpenalty\addvspace{#2}%
    \begingroup \hyphenpenalty\@M
    \@ifempty{#4}{%
      \@tempdima\csname r@tocindent\number#1\endcsname\relax
    }{%
      \@tempdima#4\relax
    }%
    \parindent\z@ \leftskip#3\relax \advance\leftskip\@tempdima\relax
    \rightskip\@pnumwidth plus4em \parfillskip-\@pnumwidth
    #5\leavevmode\hskip-\@tempdima
      \ifcase #1
       \or\or \hskip 1em \or \hskip 2em \else \hskip 3em \fi%
      #6\nobreak\relax
      \dotfill
      \hbox to\@pnumwidth{\@tocpagenum{#7}}
    \par
    \nobreak
    \endgroup
  \fi}
\makeatother

\newtheorem{theorem}{Theorem}
\numberwithin{theorem}{section}

\newtheorem{proposition}[theorem]{Proposition}

\theoremstyle{definition}

\theoremstyle{remark}
\newtheorem{remark}[theorem]{Remark}

\newcommand{\R}{\mathbb{R}}
\newcommand{\Rm}{\mathbb{R}^m}

\newcommand{\Dtwo}{\mathcal{D}_k}

\newcommand{\bx}{\boldsymbol{x}}
\newcommand{\by}{\boldsymbol{y}}
\newcommand{\bu}{\boldsymbol{u}}
\newcommand{\bv}{\boldsymbol{v}}

\newcommand{\be}{\begin{eqnarray*}}
\newcommand{\ee}{\end{eqnarray*}}
\newcommand{\bs}{\boldsymbol}

\newcommand{\Smone}{\mathbb{S}^{m-1}}

\newcommand{\Bm}{\mathbb{B}^{m}}

\newcommand{\Mk}{\mathcal{M}_k}
\newcommand{\Mkk}{\mathcal{M}_{k-1}}
\newcommand{\HK}{\mathcal{H}_k}

\newcommand{\Clm}{\mathcal{C}l_m}

\numberwithin{equation}{section}

\begin{document}

\title[Green's Formulas and Poisson's Equation for Bosonic Laplacians]{Green's Formulas and Poisson's Equation for Bosonic Laplacians}
\author[C. Ding]{Chao Ding}
\address{Department of Mathematics and Statistics, Masaryk University, Brno, Czech Republic}
\email{chaoding@math.muni.cz}
\author[J. Ryan]{John Ryan}
\address{Department of Mathematical Science, University of Arkansas, Fayetteville, AR. U.S.A.}
\email{jryan@uark.edu}
\date{\today}

\begin{abstract}
A bosonic Laplacian is a conformally invariant second order differential operator acting on smooth functions defined on domains in Euclidean space and taking values in higher order irreducible representations of the special orthogonal group. In this paper, we firstly introduce the motivation for study of the generalized Maxwell operators and bosonic Laplacians (also known as the higher spin Laplace operators). Then,  with the help of connections between Rarita-Schwinger type operators and bosonic Laplacians, we solve Poisson's equation for bosonic Laplacians. A representation formula for bounded solutions to Poisson's equation in Euclidean space is also provided. In the end, we provide Green's formulas for bosonic Laplacians in scalar-valued and Clifford-valued cases, respectively. These formulas reveal that bosonic Laplacians are self-adjoint with respect to a given $L^2$ inner product on certain compact supported function spaces. 
\medskip
	
	
\medskip
	

\end{abstract}

\maketitle
\section{Introduction}~\par
Clifford analysis is introduced as a generalization of complex function theory to the higher dimensional cases. Many researchers have successfully generalized the theory of one dimensional complex analysis to the higher dimension cases via Clifford analysis. For instance, the function theory of quaternion analysis is applied to study some boundary value problems in three or four dimensional spaces in \cite{Guer}. Many important properties and problems, such as integral formulas, series expansion, integral transforms and boundary value problems, have been generalized to higher dimensions as well in \cite{Guer,Guer1,Guer2}.
\par
The higher spin theory in Clifford analysis is the theory on functions taking values in irreducible representations of the spin group. These representation spaces are usually realized as spaces of homogeneous harmonic or monogenic (null solutions of the Dirac operator) polynomials. The study on this topic can be traced back to the work of Stein and Weiss given in \cite{SW}. Stein and Weiss introduced a technique to construct first order conformally invariant differential operators, named as Stein-Weiss gradients, with a certain type of projections. Bure\v{s}  et. al. \cite{Bures} investigated a class of generalized Rarita-Schwinger operators acting on functions taking values in irreducible representations of the spin group with weight $k+1/2$ via Clifford analysis  in 2002. These Rarita-Schwinger type operators were also studied by Dunkl et. al. \cite{Dunkl} with an analytic approach. In these two papers, one can notice that Rarita-Schwinger operators and  the Dirac operator have very similar properties, such as Cauchy's Theorem, Cauchy's integral formula, Stokes' Theorem, etc. Therefore, Rarita-Schwinger operators are also considered as generalizations of the Dirac operator in the higher spin theory. In 2016, Eelbode et. al. \cite{Eelbode} and De Bie et. al. \cite{DeBie} introduced generalizations of the Laplace operator  with respect to the conformal invariance property in the higher spin theory. These differential operators are named as  bosonic Laplacians ( the higher spin Laplace operators) or the generalized Maxwell operators, which are special cases of bosonic Laplacians. It is reasonable to expect bosonic Laplacians also have similar properties as the Laplace operator has. In \cite{DR,DWR}, the authors discovered intertwining operators, a Borel-Pompeiu formula and a Green type integral formula for bosonic Laplacians. Recently, the authors \cite{DTR} also studied Dirichlet problems for bosonic Laplacians in the upper-half space and the unit ball. Further, many important properties, such as the mean-value property, Cauchy's estimates and Liouville's Theorem, for null solutions to bosonic Laplacians have been found in \cite{DTR}. Here, we continue our investigation on properties of bosonic Laplacians in Euclidean space.
\par
\textbf{Main results:} In this paper, we firstly look into the generalized Maxwell operators by generalizing classical Maxwell equations to the higher spin spaces in Section $2$. This provides us the motivation for studying the generalized Maxwell operators and bosonic Laplacians. Some preliminaries of Clifford analysis setting, Rarita-Schwinger operators and bosonic Laplacians will also be introduced here as well.  In Section $3$, we use some properties of Rarita-Schwinger operators and their connections to bosonic Laplacians to solve Poisson's equation in the higher spin spaces. Section $4$ will be devoted to introducing Green's formulas in the higher spin spaces, which reveal that bosonic Laplacians are self-adjoint with respect to a given $L^2$ inner product. As in the harmonic anlysis, these formulas can possibly be applied in our future work on constructing Green's functions for solving certain boundary value problems.
\subsection*{Acknowledgements}
This paper is dedicated to Klaus G\"urlebeck on his 65th birthday. Chao Ding is supported by Czech Science Foundation, project GJ19-14413Y.
\section{Preliminaries}
\subsection{Notations}
Suppose that $\{e_1,\cdots,e_m\}$ is a standard orthonormal basis for the $m$-dimensional Euclidean space $\R^m$. The (real) Clifford algebra $\mathcal{C}l_m$ is generated by $\R^m$ with the relationship $e_ie_j+e_je_i=-2\delta_{i,j},\ 1\leq i,j\leq m.$ This implies that an element of the basis of the Clifford algebra can be written as $e_A=e_{j_1}\cdots e_{j_r},$ where $A=\{j_1, \cdots, j_r\}\subset \{1, 2, \cdots, m\}$ and $1\leq j_1< j_2 < \cdots < j_r \leq m.$ Hence any element $a\in \mathcal{C}l_m$ can be represented by $a=\sum_Aa_Ae_A$, where $a_A\in \mathbb{R}$. In particular, $\bs{e}_{\emptyset}=1$ and we call $a_{\emptyset}=Scr(a)$ the scalar part of $a$. The $m$-dimensional Euclidean space $\R^m$ is embedded into $\mathcal{C}l_m$ as follows.
\be
\R^m&&\longrightarrow\quad \mathcal{C}l_m,\\
(x_1,\cdots,x_m)&&\ \mapsto\quad \sum_{j=1}^mx_je_j.
\ee
Hence, for $\bx\in\R^m$, one can easily see that $\|\bx\|^2=\sum_{j=1}^mx_j^2=-\bx^2$. 
For $a=\sum_Aa_Ae_A\in\Clm$, we define the reversion of $a$ as
\begin{eqnarray*}
\widetilde{a}=\sum_{A}(-1)^{|A|(|A|-1)/2}a_Ae_A,
\end{eqnarray*}
where $|A|$ is the cardinality of $A$. In particular, $\widetilde{e_{j_1}\cdots e_{j_r}}=e_{j_r}\cdots e_{j_1}$. Also $\widetilde{ab}=\widetilde{b}\widetilde{a}$ for $a, b\in\Clm$.
\par 
Now suppose $\bs{a}\in \mathbb{S}^{m-1}\subseteq \mathbb{R}^m$ and $\bx\in\R^m$. If we consider $\bs{a}\bx\bs{a}$, we may decompose
$$\bx=\bx_{\bs{a}\parallel}+\bx_{\bs{a}\perp},$$
where $\bx_{\bs{a}\parallel}$ is the projection of $\bx$ onto $\bs{a}$ and $\bx_{\bs{a}\perp}$ is the remainder part of $\bx$ perpendicular to $\bs{a}$. Hence $\bx_{\bs{a}\parallel}$ is a scalar multiple of $a$ and we have
$$\bs{a}\bx\bs{a}=\bs{a}\bx_{\bs{a}\parallel}\bs{a}+\bs{a}\bx_{\bs{a}\perp}\bs{a}=-\bx_{\bs{a}\parallel}+\bx_{\bs{a}\perp}.$$
So the action $\bs{a}\bx\bs{a}$ describes a reflection of $x$ in the direction of $\bs{a}$. More details can be found in, for instance, \cite{Del}.
\par
The classical Dirac operator is defined as $D_{\bx}=\sum_{j=1}^m\partial_{x_j}e_j$, which factorizes the Laplace operator $\Delta_{\bx}=-D_{\bx}^2$. A $\Clm$-valued function $f(\bx)$ defined on a domain $\Omega$ in $\R^m$ is called \emph{left monogenic} if it satisfies $D_{\bx}f(\bx)=0$ in $\Omega$. Since multiplication of Clifford numbers is not commutative in general, there is a similar definition for right monogenic functions.
\subsection{Rarita-Schwinger type operators}
Let $\HK(\mathcal{C}l_m)\ (\mathcal{M}_k(\mathcal{C}l_m))$ stand for the space of Clifford-valued harmonic (monogenic) polynomials homogeneous of degree $k$. Notice that if $h_k(\bu)\in\HK(\mathcal{C}l_m)$, then $D_{\bu}h_k(\bu)\in\mathcal{M}_{k-1}(\mathcal{C}l_m)$, but $D_{\bu}\bu p_{k-1}(\bu)=(-m-2k+2)p_{k-1}(\bu),$ where $p_{k-1}(\bu)\in \Mkk(\mathcal{C}l_m)$. Hence, we have
\begin{eqnarray}\label{Almansi}
\mathcal{H}_k(\mathcal{C}l_m)=\mathcal{M}_k(\mathcal{C}l_m)\oplus \bu\mathcal{M}_{k-1}(\mathcal{C}l_m),\ h_k=p_k+\bu p_{k-1}.
\end{eqnarray}
This is called an \emph{Almansi-Fischer decomposition} of $\HK(\mathcal{C}l_m)$ \cite{Dunkl}. In this decomposition, we have $P_k^+$ and $P_k^-$ as the projection maps 
\begin{eqnarray*}
&&P_k^+=1+\frac{\bu D_{\bu}}{m+2k-2}:\ \mathcal{H}_k(\mathcal{C}l_m)\longrightarrow \mathcal{M}_k(\mathcal{C}l_m), \label{Pk+}\\
&&P_k^-=I-P_k^+=\frac{-\bu D_{\bu}}{m+2k-2}:\ \mathcal{H}_k(\mathcal{C}l_m)\longrightarrow \bu\mathcal{M}_{k-1}(\mathcal{C}l_m) \label{Pk-}.
\end{eqnarray*}
Suppose $\Omega$ is a domain in $\mathbb{R}^m$. Consider a differentiable function $f: \Omega\times \mathbb{R}^m\longrightarrow \mathcal{C}l_m$
such that, for each $\bx\in \Omega$, $f(\bx,\bu)$ is a left monogenic polynomial homogeneous of degree $k$ in $\bu$. Then \textbf{the Rarita-Schwinger operator} \cite{Bures,Dunkl} is defined by 
 $$R_k=P_k^+D_{\bx}:\ C^{\infty}(\Rm,\Mk(\mathcal{C}l_m))\longrightarrow C^{\infty}(\Rm,\Mk(\mathcal{C}l_m)).$$
 We also need the following three Rarita-Schwinger type operators.
 \begin{eqnarray*}
	&&\text{\textbf{The twistor operator:}}\\
	&& T_k=P_k^+D_{\bx}:\ C^{\infty}(\Rm,\bu\Mkk(\mathcal{C}l_m))\longrightarrow C^{\infty}(\Rm,\Mk(\mathcal{C}l_m)),\\
	&&\text{\textbf{The dual twistor operator:}}\\ 
	&&T_k^*=P_k^-D_{\bx}:\ C^{\infty}(\Rm,\Mk(\mathcal{C}l_m))\longrightarrow C^{\infty}(\Rm,\bu\Mkk(\mathcal{C}l_m)),\\
	&&\text{\textbf{The remaining operator:}}\\
	&& Q_k=P_k^-D_{\bx}:\ C^{\infty}(\Rm,\bu\Mkk(\mathcal{C}l_m))\longrightarrow C^{\infty}(\Rm,\bu\Mkk(\mathcal{C}l_m)).
	\end{eqnarray*}
	More details can be found in \cite{Bures,Dunkl}.
\subsection{The generalized Maxwell operators}
In \cite{Eelbode}, the authors constructed the generalized Maxwell operators as a type of second order conformally invariant differential operators on particular function spaces. It was also pointed out that these differential operators reduced to the classical source-free Maxwell equations in the Minkowski space. Here, we will show the details that how to obtain the generalized Maxwell operators from the classical source-free Maxwell equations. This gives us the motivation to study these particular type of second order differential operators.
\par 
Recall that the classical source-free coupled Maxwell equations are given by
\be
&&\nabla\cdot\bs{E}=0,\ \nabla\times\bs{E}=-\frac{\partial{\bs{B}}}{\partial t},\\
&&
\nabla\cdot\bs{B}=0,\ \nabla\times\bs{B}=\mu_0\epsilon_0\frac{\partial\bs{E}}{\partial t},
\ee
Where $\bs{E}$ stands for the electric field, $\bs{B}$ stands for the magnetic field, $\bs{B}$ and $\bs{E}$ are both vector fields in $\mathbb{R}^3$. $\mu_0$ is the permeability of free space and $\epsilon_0$ is the permittivity of free space. 
\par
Since $\nabla\cdot\bs{B}=0$, we can define $\bs{B}$ in terms of a vector potential $\bs{C}$ as $\bs{B}=\nabla\times\bs{C}$. From Maxwell Faraday's equation, one obtains that $\nabla\times(\bs{E}+\frac{\partial\bs{C}}{\partial t})=0$. This means that $\bs{E}+\frac{\partial\bs{C}}{\partial t}$ can be written as the gradient of some scalar function, namely, a scalar potential $\Phi$. Hence, one has
\begin{eqnarray}\label{Electric}
\bs{E}=-\nabla\Phi-\frac{\partial\bs{C}}{\partial t}
\end{eqnarray}
with a scalar potential $\Phi$. 
\par
Plugging $\bs{B}=\nabla\times\bs{C}$ and $\bs{E}=-\nabla\Phi-\frac{\partial\bs{C}}{\partial t}$ into $\nabla\cdot\bs{E}=0,\ \nabla\times\bs{B}=\mu_0\epsilon_0\frac{\partial\bs{E}}{\partial t}$, one has

\begin{eqnarray}\label{newmax}
&&\nabla^2\Phi+\frac{\partial}{\partial t}(\nabla\cdot\bs{C})=0,\nonumber\\
&&\nabla^2\bs{C}-\mu_0\epsilon_0\frac{\partial^2\bs{C}}{\partial^2 t}-\nabla(\nabla\cdot\bs{C}+\mu_0\epsilon_0\frac{\partial\Phi}{\partial t})=0.
\end{eqnarray}
Next, we choose a set of potentials $(\bs{C},\Phi)$ to satisfy the Lorenz condition $\nabla\cdot \bs{C}+\mu_0\epsilon_0\frac{\partial\Phi}{\partial t}=0$, which uncouples the pair of equations given in (\ref{newmax}). Now one has
\be
\nabla^2\Phi-\mu_0\epsilon_0\frac{\partial^2\Phi}{\partial t^2}=0,\ \nabla^2\bs{C}-\mu_0\epsilon_0\frac{\partial^2\bs{C}}{\partial t^2}=0.
\ee
The potentials $\bs{C}$ and $\Phi$ form a 4-vector potential $\bs{C}^{\alpha}=(\Phi,\bs{C})$. Then the two equations above and the Lorenz condition becomes
\be
\square C^{\alpha}=0,\ \partial_{\alpha}A^{\alpha}=0,
\ee
where $\square=\nabla^2-\mu_0\epsilon_0\frac{\partial^2}{\partial t^2}$ is the d'Alembert operator. For $\bs{E}=-\nabla\Phi-\frac{\partial\bs{C}}{\partial t}$ and $\bs{B}=\nabla\times\bs{C}$, if the space coordinate is given by $(x,y,z)$, the $x$ components of $\bs{E}$ and $\bs{B}$ can be written explicitly as
\be
&&E_x=-\mu_o\epsilon_0\frac{\partial C_x}{\partial t}-\frac{\partial \Phi}{\partial x}=-(\partial^0C^1-\partial^1C^0),\\
&&B_x=\frac{\partial C_z}{\partial y}-\frac{\partial C_y}{\partial z}=-(\partial^2C^3-\partial^3C^2),
\ee
where $\partial^{\alpha}=(\frac{\partial}{\partial t},-\nabla)$. These equations imply that the electric and magnetic fields are the elements of a second rank, antisymmetric field-strength tensor $F^{\alpha\beta}=\partial^{\alpha}C^{\beta}-\partial^{\beta}C^{\alpha}$.
\par
This leads to a different formulation of Maxwell equations in terms of a differential 2-form $F^{\alpha\beta}$, known as the Maxwell-Faraday tensor, $\partial_{\alpha}F^{\alpha\beta}=0$, which can also be rewritten as
\be
\square C^{\beta}-\partial^{\beta}\partial_{\alpha}C^{\alpha}=0,
\ee
where $C^{\alpha}=(\Phi,\bs{C})$, $\partial^{\alpha}=(\partial_t,-\nabla)$ and $\partial_{\alpha}=(\partial_t,\nabla)$ and 
\[
F^{\alpha\beta}
=
\begin{pmatrix}
    0 & E_x & E_y &  E_z \\
    -E_x & 0 & -B_z &  B_y \\
    -E_y & B_z & 0 & -B_x \\
   -E_z & -B_y & B_x  & 0
\end{pmatrix}.
\]
The equation $\square C^{\beta}-\partial_{\beta}\partial^{\alpha}C^{\alpha}=0$ in the Minkowski space has a generalization to the $m$-dimensional Euclidean space given below.
\be
\Delta_{\bs{x}}\bs{f_s}(\bs{x})-\frac{4}{m}\sum_{j=1}^{m}\partial_{x_s}\partial_{x_j}\bs{f_j}(\bs{x})=0,\ 1\leq s\leq m.
\ee
where $f_s(\bx)$ is a vector valued function, with $\bx\in\R^m$. The constant $4/m$ allows the generalized operator above to preserve the conformal invariance property of the Maxwell equations. Notice that the space of real-valued homogeneous of harmonic polynomials with degree-$1$ with respect to a variable $\bu\in\R^m$ , denoted by $\mathcal{H}_1$ , is spanned by $\{u_1,\cdots,u_m\}$. The $m$ equations above can be replaced by one equation
\be
\sum_{s=1}^{m}u_s\bigg(\Delta_{\bs{x}}\bs{f_s}(\bs{x})-\frac{4}{m}\sum_{j=1}^{m}\partial_{x_s}\partial_{x_j}\bs{f_j}(\bs{x})\bigg)=0.
\ee
This equation can also be rewritten as
\be
0&=&\sum_{s=1}^m\Delta_{\bx}u_sf_s(\bx)-\frac{4}{m}\sum_{j,s=1}^mu_s\partial_{x_s}\partial_{x_j}f_j(\bx)\\
&=&\sum_{s=1}^m\Delta_{\bx}u_sf_s(\bx)-\frac{4}{m}\sum_{j,s,k=1}^mu_s\partial_{x_s}\partial_{x_j}\partial_{u_j}u_kf_k(\bx).
\ee
Now, we consider $f(\bx,\bu):=\sum_{s=1}^mu_sf_s(\bx)$ as a $\mathcal{H}_1(\R)$-valued function on $\R^m$. With the help of the Dirac operator $\bs{D}_{\bx}=\sum_{s=1}^m\bs{e}_s\partial_{x_s}$ in the $m$-dimensional Euclidean space. This equation can be written in a compact form
\be
\left(\Delta_{\bs{x}}-\frac{4}{m}\langle \bs{u},\bs{D}_{\bs{x}}\rangle\langle \bs{D}_{\bs{u}},\bs{D}_{\bs{x}}\rangle\right)\bs{f}(\bs{x},\bs{u})=0,
\ee
where $\langle\ ,\ \rangle$ is the standard inner product in Euclidean space.
The operator $\Delta_{\bs{x}}-\frac{4}{m}\langle \bs{u},\bs{D}_{\bs{x}}\rangle\langle \bs{D}_{\bs{u}},\bs{D}_{\bs{x}}\rangle$, denoted by $\mathcal{D}_1$ is called the generalized Maxwell operator, which is a second order conformally invariant differential operator in the higher spin spaces in Clifford analysis. It was firstly constructed by Eelbode et. al. in \cite{Eelbode}.
\par
More generally, if we consider a function $f(\bx,\bu)\in C^{\infty}(\R^m,\HK(\R))$, i.e., for a fixed $\bx\in\R^m$, $f(\bx,\bu)\in\HK(\R)$ with respect to $\bu$. The second order conformally invariant differential operators, named as bosonic Laplacians (also known as the higher spin Laplace operators \cite{Eelbode}), are defined as 
\begin{eqnarray}\label{Dtwo}
&&\Dtwo:\ C^{\infty}(\R^m,\HK(\R))\longrightarrow C^{\infty}(\R^m,\HK(\R)),\nonumber\\
&&\Dtwo=\Delta_{\bx}-\frac{4\langle \bu,D_{\bx}\rangle\langle D_{\bu},D_{\bx}\rangle}{m+2k-2}+\frac{4|u|^2\langle D_{\bu},D_{\bx}\rangle^2}{(m+2k-2)(m+2k-4)},
\end{eqnarray}
where $\langle\ ,\ \rangle$ is the standard inner product in $\R^m$. One can easily see that $\Dtwo$ reduces to the generalized Maxwell operator when $k=1$.

 \section{Poisson's equation and representation formula}
In this section, we will review some properties of the Rarita-Schwinger type operators and bosonic Laplacians from \cite{DeBie,DR,Dunkl,Li}, which will be needed for solving a Poisson's equation in the higher spin spaces.
\par
 Let $Z_k^1(\bu,\bv)$ be the reproducing kernel for $\Mk(\mathcal{C}l_m)$, which satisfies
\begin{eqnarray*}
f(\bv)=\int\displaylimits_{\Smone}Z_k^1(\bu,\bv)f(\bu)dS(\bu),\ for\ all\ f(\bv)\in\Mk(\mathcal{C}l_m).
\end{eqnarray*}
Then the fundamental solution for $R_k$ (Section $5.1$, \cite{Bures}) is 
\begin{eqnarray*}
E_k(\bx,\by,\bu,\bv)=\displaystyle\frac{m+2k-2}{(m-2)\omega_{m}}\frac{\by-\bx}{|\by-\bx|^m}Z_k^1\bigg(\frac{(\by-\bx)\bu(\by-\bx)}{|\by-\bx|^2},\bv\bigg).
\end{eqnarray*}
Similarly, we have the fundamental solution for $Q_k$ (Section $3$, \cite{Li}) as follows.
\begin{eqnarray*}
F_k(\bx,\by,\bu,\bv)=\displaystyle\frac{m+2k-2}{(2-m)\omega_{m}}\bu\frac{\by-\bx}{|\by-\bx|^m}Z_{k-1}^1\bigg(\frac{(\by-\bx)\bu(\by-\bx)}{|\by-\bx|^2},\bv\bigg)\bv.
\end{eqnarray*}
We also have that
\begin{theorem}[Theorem 10, \cite{Dunkl}]
Let $f\in C_c^{\infty}(\R^m,\mathcal{M}_k(\mathcal{C}l_m))$, we have
\be
R_k\int\displaylimits_{\R^m}\int\displaylimits_{\Smone}E_k(\bx,\by,\bu,\bv)f(\bx,\bu)dS(\bu)d\bx=f(\by,\bv),
\ee
\end{theorem}
and
\begin{theorem}[Theorem 6, \cite{Li}]
Let $f\in C_c^{\infty}(\R^m,\bu\mathcal{M}_{k-1}(\mathcal{C}l_m))$, we have
\be
Q_k\int\displaylimits_{\R^m}\int\displaylimits_{\Smone}F_k(\bx,\by,\bu,\bv)f(\bx,\bu)dS(\bu)d\bx=f(\by,\bv).
\ee
\end{theorem}
Recall that (Proposition 1, \cite{DR}) the connection between bosonic Laplacians and the Rarita-Schwinger type operators is given by
\begin{align*}
&\Dtwo =-R_k^2P_k^++\frac{2R_kT_kP_k^-}{m+2k-4}-\frac{2Q_kT_k^*P_k^+}{m+2k-4}-\frac{(m+2k)Q_k^2P_k^-}{m+2k-4},\\
=&R_k\bigg[-R_kP_k^++\frac{2T_kP_k^-}{m+2k-4}\bigg]+Q_k\bigg[-\frac{2T_k^*P_k^+}{m+2k-4}-\frac{(m+2k)Q_kP_k^-}{m+2k-4}\bigg].
\end{align*}
We let $A_k=-R_kP_k^++\displaystyle\frac{2T_kP_k^-}{m+2k-4}$ and $B_k=-\displaystyle\frac{2T_k^*P_k^+}{m+2k-4}-\displaystyle\frac{(m+2k)Q_kP_k^-}{m+2k-4}$ for convenience, then we have
\begin{eqnarray}\label{connection}
\Dtwo=R_kA_k+Q_kB_k.
\end{eqnarray}
Now, we can solve a Poisson's equation for bosonic Laplacian $\Dtwo$ as follows.
\begin{theorem}[Solving Poisson's equation]\label{PoissonEqn}
Let $f\in C^2_c(\R^m\times\Bm,\HK(\mathbb{R}))$, that is $f\in C^2(\R^m\times\Bm,\HK(\mathbb{R}))$ and $f$ has compact support with respect to $\bx$, and set
\begin{eqnarray}
\Phi(\by,\bv)=\int\displaylimits_{\R^m}\int\displaylimits_{\Smone}H_k(\bx,\by,\bu,\bv)f(\bx,\bu)dS(\bu)d\bx.
\end{eqnarray}
Then, we have
\begin{enumerate}
\item $\Phi\in C^2(\R^m\times\Bm,\HK(\mathbb{R}))$,
\item $\Dtwo\Phi=f\ \text{in}\ \R^m\times\Bm$.
\end{enumerate}
\end{theorem}
\begin{proof}
1.Firstly, we notice that $Z_k\bigg(\displaystyle\frac{\bx\bu\bx}{|\bx|^2},\bv\bigg)$ is a $k$-homogeneous harmonic polynomial with respect to $\bv$, so $\Phi$ is harmonic with respect to $\bv$. Further, we have
\be
\Phi(\by,\bv)&=&\int\displaylimits_{\R^m}\int\displaylimits_{\Smone}H_k(\bx,\by,\bu,\bv)f(\bx,\bu)dS(\bu)d\bx\\
&=&\int\displaylimits_{\R^m}\int\displaylimits_{\Smone}H_k(\bx,0,\bu,\bv)f(\by-\bx,\bu)dS(\bu)d\bx,
\ee
hence,
\be
&&\frac{\Phi(\by+h\bs{e}_j,\bv)-\Phi(\by,\bv)}{h}\\
&=&\int\displaylimits_{\R^m}\int\displaylimits_{\Smone}H_k(\bx,0,\bu,\bv)\frac{f(\by-\bx+h\bs{e}_j,\bu)-f(\by-\bx,\bu)}{h}dS(\bu)d\bx,
\ee
where $h\neq 0$ and $\bs{e}_j=(0,\cdots,1,\cdots,0)$ with $1$ in the $j^{th}$ spot. Notice that $f\in C^2_c(\R^m\times\Bm,\HK(\mathbb{R}))$, which tells us that
\be
\frac{f(\by-\bx+h\bs{e}_j,\bu)-f(\by-\bx,\bu)}{h}\rightarrow f_{y_j}(\by-\bx,\bu),\ h\rightarrow 0
\ee
uniformly. Hence, we have 
\be
\Phi_{y_j}(\by,\bv)=\int\displaylimits_{\R^m}\int\displaylimits_{\Smone}H_k(\bx,0,\bu,\bv)f_{y_j}(\by-\bx,\bu)dS(\bu)d\bx.
\ee
A similar argument can be applied to the second derivatives, which implies $\Phi\in C^2(\R^m\times\Bm,\HK(\mathbb{R}))$.
\par
2. We prove the second claim by applying the properties of the Rarita-Schwinger type operators and the connection between the Rarita-Schwinger type operators and bosonic Laplacians given in \eqref{connection}.
\par
Recall that $\Dtwo=R_kA_k+Q_kB_k$, and from Proposition 2 in \cite{DR}, we know that 
\be
&&R_kA_kH_k(\bx,\by,\bu,\bv)=R_kE_k(\bx,\by,\bu,\bv)=\delta(\by-\bx)Z_k^1(\bu,\bv),\\
&&Q_kB_kH_k(\bx,\by,\bu,\bv)=R_kE_k(\bx,\by,\bu,\bv)=\delta(\by-\bx)\bv Z_{k-1}^1(\bu,\bv)\bu,
\ee
in the distributional sense. Further, it also tells us that
\begin{align*}
&\Dtwo\int\displaylimits_{\R^m}\int\displaylimits_{\Smone}H_k(\bx,\by,\bu,\bv)f(\bx,\bu)dS(\bu)d\bx\\
=&(R_kA_k+Q_kB_k)\int\displaylimits_{\R^m}\int\displaylimits_{\Smone}H_k(\bx,\by,\bu,\bv)(P_k^++P_k^-)f(\bx,\bu)dS(\bu)d\bx\\
=&R_kA_k\int\displaylimits_{\R^m}\int\displaylimits_{\Smone}H_k(\bx,\by,\bu,\bv)P_k^+f(\bx,\bu)dS(\bu)d\bx\\
&+Q_kB_k\int\displaylimits_{\R^m}\int\displaylimits_{\Smone}H_k(\bx,\by,\bu,\bv)P_k^-f(\bx,\bu)dS(\bu)d\bx.
\end{align*}
The last equation comes from the fact that 
\be
R_kA_kH_k,P_k^+f\in\Mk(\mathcal{C}l_m);\ Q_kB_kH_k,P_k^-f\in\bu\Mkk(\Clm),
\ee
and two functions from different function spaces above are orthogonal to each other with respect to the integral over the unit sphere with respect to $\bu$ (Lemma $5$ in \cite{Dunkl}). Further, \cite[Theorem 10]{Dunkl} shows that 
\be
&&R_kA_k\int\displaylimits_{\R^m}\int\displaylimits_{\Smone}H_k(\bx,\by,\bu,\bv)P_k^+f(\bx,\bu)dS(\bu)d\bx\\
&=&\int\displaylimits_{\R^m}\int\displaylimits_{\Smone}R_kE_k(\bx,\by,\bu,\bv)P_k^+f(\bx,\bu)dS(\bu)d\bx=P^+_kf(\by,\bv),
\ee
and \cite[Theorem 6]{Li} gives us that
\be
Q_kB_k\int\displaylimits_{\R^m}\int\displaylimits_{\Smone}H_k(\bx,\by,\bu,\bv)P_k^-f(\bx,\bu)dS(\bu)d\bx=P^-_kf(\by,\bv).
\ee
Hence, one has 
\be
&&\Dtwo\int\displaylimits_{\R^m}\int\displaylimits_{\Smone}H_k(\bx,\by,\bu,\bv)f(\bx,\bu)dS(\bu)d\bx\\
&=&P_k^+f(\by,\bv)+P_k^-f(\by,\bv)=f(\by,\bv),
\ee
which completes the proof.
\end{proof}
With Theorem \ref{PoissonEqn} and the Liouville-type theorem given in \cite[Theorem 5.7]{DTR}, one can immediately have a representation formula as follows.
\begin{theorem}[Representation formula] Assume  $m>4$ and $f\in C^2_c(\R^m\times\Bm,\HK(\mathbb{R}))$. Then any bounded solution of $\Dtwo g=f$ in $\R^m\times\Bm$ has the form
\be
g(\by,\bv)=\int\displaylimits_{\R^m}\int\displaylimits_{\Smone}H_k(\bx,\by,\bu,\bv)f(\bx,\bu)dS(\bu)d\bx+h(\bv), \quad \by\in\R^m,\, \bv\in\Bm,
\ee
where $h\in\HK(\mathbb{R})$.
\end{theorem}
\begin{proof}
Notice that $H_k(\bx,\by,\bu,\bv)\rightarrow 0$ when $|\bx-\by|\rightarrow \infty$ for $m>4$, hence, 
\be
\Phi(\by,\bv):=\int\displaylimits_{\R^m}\int\displaylimits_{\Smone}H_k(\bx,\by,\bu,\bv)f(\bx,\bu)dS(\bu)d\bx
\ee
is a bounded solution in $C^2(\R^m\times\Bm,\HK(\mathbb{R}))$ for $\Dtwo g=f$ in $\R^m\times\Bm$. If $\Phi'$ is another bounded solution in $\R^m\times\Bm$, then we have a bounded solution $\Phi-\Phi'$ for $\Dtwo g=0$ in $\R^m\times\Bm$. In accordance to the Liouville-type theorem, $\Phi-\Phi'=h(\bv)$ in $\R^m\times\Bm$, where $h(\bv)\in\HK(\mathbb{R})$, which completes the proof.
\end{proof}
In particular, if we let $g(\bx,\bu)$ above be the fundamental solution of $\Dtwo$ given by (Theorem $5.2$, \cite{DeBie})
\begin{eqnarray*}
H_k(\bx,\by,\bu,\bv)=c_{m,k}|\by-\bx|^{2-m}Z_k\bigg(\displaystyle\frac{(\by-\bx)\bu(\by-\bx)}{|\by-\bx|^2},v\bigg),
\end{eqnarray*}
and discuss the only singular point $\by$ with the technique applied in \cite[Theorem 6]{DR}, we can obtain a Green's integral formula as follows.
\begin{proposition}[Green's integral formula]
Let $\Omega\subset\R^m$ be an open domain and $f\in C^2(\Omega\times\Bm,\HK(\R))$. Then, we have
\be
f(\by,\bv)
&=&\int\displaylimits_{\partial\Omega}\int\displaylimits_{\Smone}(Af)(\bx,\bu)H_k(\bx,\by,\bu,\bv)\\
&&\quad\quad\quad-f(\bx,\bu)(AH_k)(\bx,\by,\bu,\bv)dS(\bu)d\sigma(\bx),
\ee
where $A$ is the operator given in Theorem \ref{Green}.
\end{proposition}
\section{Green's formulas in the higher spin spaces}
It is well-known that Green's formulas are very important in many applications, especially in the study of boundary value problems. In this section, we provide Green's formulas for bosonic Laplacians acting on scalar-valued functions and Clifford-valued functions, respectively.
\subsection*{Green's formula: scalar-valued version}
\hfill \\
Let $\Omega\subset\R^m$ be an open domain, we consider scalar-valued function spaces $C^2(\Omega\times\Bm,\HK(\R))$ and define the following inner product 
\begin{eqnarray}\label{innerproduct}
\langle f\ |\ g\rangle=\int\displaylimits_{\Omega}\int\displaylimits_{\Smone}f(\bx,\bu)g(\bx,\bu)dS(\bu)d\bx,\quad f,g\in C^2(\Omega\times\Bm,\HK(\R)).
\end{eqnarray}
\begin{theorem}[Green's formula: scalar-valued version]\label{Green}
Let $\Omega\subset\R^m$ be an open domain and $f,g\in C^2(\Omega\times\Bm,\HK(\R))$. Then, we have
\be
&&\int\displaylimits_{\Omega}\int\displaylimits_{\Smone}(\Dtwo f(\bx,\bu))g(\bx,\bu)-f(\bx,\bu)(\Dtwo g(\bx,\bu))dS(\bu)d\bx\\
&=&\int\displaylimits_{\partial\Omega}\int\displaylimits_{\Smone}(Af)(\bx,\bu)g(\bx,\bu)-f(\bx,\bu)(Ag)(\bx,\bu)dS(\bu)d\sigma(\bx),
\ee
 where $\sigma(x)$ is the area element on $\partial\Omega$ and $n_{\bx}$ is the outward unit normal vector on $\partial\Omega$ and
 \be
 A=\frac{\partial}{\partial n_{\bx}}-\frac{4\langle\bu,n_{\bx}\rangle\langle D_{\bu},D_{\bx}\rangle}{m+2k-2}.
 \ee
\end{theorem}
\begin{proof}
The main tools used here are Stokes' Theorem and the orthogonality between homogeneous harmonic polynomials with respect to the $L^2$ integral inner product over the unit sphere. Firstly, let us look at
\begin{align}\label{ezero}
&\int\displaylimits_{\Omega}\int\displaylimits_{\Smone}(\Dtwo f(\bx,\bu))g(\bx,\bu)dS(\bu)d\bx\nonumber\\
=&\int\displaylimits_{\Omega}\int\displaylimits_{\Smone}\bigg[\bigg(\Delta_{\bx}-\frac{4\langle \bu,D_{\bx}\rangle\langle D_{\bu},D_{\bx}\rangle}{m+2k-2}+\frac{4|u|^2\langle D_{\bu},D_{\bx}\rangle^2}{(m+2k-2)(m+2k-4)}\bigg) f(\bx,\bu)\bigg]\nonumber\\
&\quad\quad\quad\quad\quad\quad\quad\quad\quad\quad\quad\quad\quad\quad\quad\quad\quad\quad\quad\quad\quad\quad\quad\quad \cdot g(\bx,\bu)dS(\bu)d\bx\nonumber\\
=&\int\displaylimits_{\Omega}\int\displaylimits_{\Smone}\bigg[\bigg(\Delta_{\bx}-\frac{4\langle \bu,D_{\bx}\rangle\langle D_{\bu},D_{\bx}\rangle}{m+2k-2}\bigg) f(\bx,\bu)\bigg]g(\bx,\bu)dS(\bu)d\bx.
\end{align}
the last equation comes from the fact that
\be
&&\int\displaylimits_{\Smone}|\bu|^2\langle D_{\bu},D_{\bx}\rangle^2f(\bx,\bu)g(\bx,\bu)dS(\bu)\\
&=&\int\displaylimits_{\Smone}\langle D_{\bu},D_{\bx}\rangle^2f(\bx,\bu)g(\bx,\bu)dS(\bu)=0,
\ee
where $\langle D_{\bu},D_{\bx}\rangle^2f\in\mathcal{H}_{k-2}$, $g\in\HK$, hence they are orthogonal to each other with respect to the integral over the unit sphere. On the one hand,
 by Green's formula with respect to $\bx$, we have
 \begin{align}\label{eone}
& \int\displaylimits_{\Omega}\int\displaylimits_{\Smone}\Delta_{\bx}f(\bx,\bu)g(\bx,\bu)dS(\bu)d\bx=\int\displaylimits_{\Omega}\int\displaylimits_{\Smone}f(\bx,\bu)\Delta_{\bx}g(\bx,\bu)dS(\bu)d\bx\nonumber\\
 &+ \int\displaylimits_{\partial\Omega}\int\displaylimits_{\Smone}\frac{\partial f(\bx,\bu)}{\partial n_x}g(\bx,\bu)dS(\bu)d\sigma(\bx)
 - \int\displaylimits_{\partial\Omega}\int\displaylimits_{\Smone} f(\bx,\bu)\frac{\partial g(\bx,\bu)}{\partial n_x}dS(\bu)d\sigma(\bx).
 \end{align}
 On the other hand, we have
 \be
&& \int\displaylimits_{\Omega}\int\displaylimits_{\Smone}\langle \bu,D_{\bx}\rangle\langle D_{\bu},D_{\bx}\rangle f(\bx,\bu)g(\bx,\bu)dS(\bu)d\bx\\
&=&-\int\displaylimits_{\Omega}\int\displaylimits_{\Smone}\langle D_{\bu},D_{\bx}\rangle f(\bx,\bu)\langle\bu,D_{\bx}\rangle g(\bx,\bu)dS(\bu)d\bx\\
&&+\int\displaylimits_{\partial\Omega}\int\displaylimits_{\Smone}\langle\bu,n_x\rangle\langle D_{\bu},D_{\bx}\rangle f(\bx,\bu)g(\bx,\bu)dS(\bu)d\sigma(\bx).
\ee
Noticing that $\bu$ is the unit normal vector on the unit sphere, we apply Stokes' Theorem with respect to $\bu$ to the first integral above, which is equal to
\be
&&-\int\displaylimits_{\Omega}\int\displaylimits_{\Bm}\sum_{j=1}^m\partial_{u_j}\langle D_{\bu},D_{\bx}\rangle f(\bx,\bu)\partial_{x_j} g(\bx,\bu)d\bu d\bx\\
&-&\int\displaylimits_{\Omega}\int\displaylimits_{\Bm}\langle D_{\bu},D_{\bx}\rangle f(\bx,\bu)\langle D_{\bu},D_{\bx}\rangle g(\bx,\bu)d\bu d\bx\\
 &=&-\int\displaylimits_{\Omega}\int\displaylimits_{\Bm}\langle D_{\bu},D_{\bx}\rangle f(\bx,\bu)\langle D_{\bu},D_{\bx}\rangle g(\bx,\bu)d\bu d\bx.
 \ee
 where the equation above comes from the fact that $\partial_{u_j}\langle D_{\bu},D_{\bx}\rangle f(\bx,\bu)\in\mathcal{H}_{k-2}$ and $\partial_{x_j} g(\bx,\bu)\in\HK$, which implies that they are orthogonal to each other. Now, we apply Stokes' Theorem to $\bx$ and $\bu$ separately and also use the orthogonality between $\HK$ and $\mathcal{H}_{k-2}$, we eventually obtain
 \begin{eqnarray}\label{etwo}
&& \int\displaylimits_{\Omega}\int\displaylimits_{\Smone}\langle \bu,D_{\bx}\rangle\langle D_{\bu},D_{\bx}\rangle f(\bx,\bu)g(\bx,\bu)dS(\bu)d\bx\nonumber\\
 &=& \int\displaylimits_{\Omega}\int\displaylimits_{\Smone}\langle \bu,D_{\bx}\rangle\langle D_{\bu},D_{\bx}\rangle f(\bx,\bu)g(\bx,\bu)dS(\bu)d\bx\nonumber\\
 &+& \int\displaylimits_{\partial\Omega}\int\displaylimits_{\Smone}\langle \bu,n_{\bx}\rangle\langle D_{\bu},D_{\bx}\rangle f(\bx,\bu)g(\bx,\bu)dS(\bu)d\sigma(\bx)\nonumber\\
 &-&\int\displaylimits_{\partial\Omega}\int\displaylimits_{\Smone} f(\bx,\bu)\langle \bu,n_{\bx}\rangle\langle D_{\bu},D_{\bx}\rangle g(\bx,\bu)dS(\bu)d\sigma(\bx).
 \end{eqnarray}
 Now, we plug \eqref{eone} and \eqref{etwo} into \eqref{ezero}, which completes the proof.
\end{proof}
\begin{remark}
The Green's formula above shows that bosonic Laplacians $\Dtwo$ are self-adjoint with respect to the inner product $\langle\ |\ \rangle$ given in \eqref{innerproduct} for functions vanishing on the boundary of the domain with respect to $\bx$.
\end{remark}
\subsection*{Green's formula: Clifford-valued version}
\hfill\\
The difficulty in the Clifford-valued case is caused by the fact that multiplication of Clifford numbers is not commutative in general. In our case, we need to use the connection between Rarita-Schwinger type operators and bosonic Laplacians given in \eqref{connection} and Stokes' Theorems for Rairta-Schwinger type operators given in \cite[Theorem 4,6,8,10]{DR1}.
\begin{theorem}[Green's formula: Clifford-valued version]
Let $\Omega\subset\R^m$ be a bounded domain, we assume that $f,g\in C^2(\Omega\times\Bm,\HK(\mathcal{C}l_m))$, then we have
\begin{align*}
&\int_{\Omega}\int_{\Smone}\Dtwo f(\bx,\bu)g(\bx,\bu)dS(\bu)d\bx\\
=&\int_{\Omega}\int_{\Smone} f(\bx,\bu)\Dtwo g(\bx,\bu)dS(\bu)d\bx\\
&+\int_{\partial\Omega}\int_{\Smone}fP_{k,r}^+d\sigma_{\bx}\bigg(-R_kP_k^++\frac{2Q_kP_k^-}{m+2k-4}\bigg)gdS(\bu)\\
&+\int_{\partial\Omega}\int_{\Smone}fP_{k,r}^-d\sigma_{\bx}\bigg(-\frac{2R_kP_k^+}{m+2k-4}+\frac{(m+2k)Q_kP_k^-}{m+2k-4}\bigg)gdS(\bu)\\
&+\int_{\partial\Omega}\int_{\Smone}f\bigg(P_{k,r}^+R_{k,r}+\frac{2P_{k,r}^-T_{k,r}}{m+2k-4}\bigg)d\sigma_{\bx}P_k^+gdS(\bu)\\
&+\int_{\partial\Omega}\int_{\Smone}f\bigg(-\frac{2P_{k,r}^+T_{k,r}^*}{m+2k-4}-\frac{(m+2k)P_{k,r}^-Q_{k,r}}{m+2k-4}\bigg)d\sigma_{\bx}P_k^-gdS(\bu),
\end{align*}
where $P_{k,r}$ stands for the projection operator $P_k$ acting from the right hand side.
\end{theorem}
\begin{proof}
Recall that
\begin{align*}
\Dtwo &=-R_k^2P_k^++\frac{2R_kT_kP_k^-}{m+2k-4}-\frac{2Q_kT_k^*P_k^+}{m+2k-4}-\frac{(m+2k)Q_k^2P_k^-}{m+2k-4}\\
&=-R_k^2P_k^++\frac{2T_k^*R_kP_k^+}{m+2k-4}-\frac{2T_kQ_kP_k^-}{m+2k-4}-\frac{(m+2k)Q_k^2P_k^-}{m+2k-4}.
\end{align*}
Then, we will calculate 
\begin{eqnarray}\label{id0}
\int_{\Omega}\int_{\Smone}\Dtwo f(\bx,\bu)g(\bx,\bu)dS(\bu)d\bx
\end{eqnarray}
 as the sum of four integrals regarding the four terms in the first expression of $\Dtwo$ above. In the calculation below, we write the functions $f$ without the variables for convenience. Firstly, we apply \cite[Lemma 5]{Dunkl} and Stokes' Theorem for $R_k$ \cite[Theorem 4]{DR1} to obtain 
\begin{align}\label{id1}
&\int_{\Omega}\int_{\Smone}fP_{k,r}^+R_{k,r}^2gdS(\bu)d\bx
=\int_{\Omega}\int_{\Smone}fP_{k,r}^+R_{k,r}^2P_k^+gdS(\bu)d\bx\nonumber\\
=&-\int_{\Omega}\int_{\Smone}fP_{k,r}^+R_{k,r}R_kP_k^+gdS(\bu)d\bx+\int_{\partial\Omega}\int_{\Smone}fP_{k,r}^+R_{k,r}d\sigma_{\bx}P_k^+gdS(\bu)\nonumber\\
=&\int_{\Omega}\int_{\Smone}fP_{k,r}^+R_k^2P_k^+gdS(\bu)d\bx-\int_{\partial\Omega}\int_{\Smone}fP_{k,r}^+d\sigma_{\bx}R_kP_k^+gdS(\bu)\nonumber\\
&+\int_{\partial\Omega}\int_{\Smone}fP_{k,r}^+R_{k,r}d\sigma_{\bx}P_k^+gdS(\bu).
\end{align}
With the help of Stokes' Theorem for $R_k$ \cite[Theorem 4]{DR1} and Stokes' Theorem for $T_k$ \cite[Theorem 10]{DR1}, we can also have
\begin{align}\label{id2}
&\int_{\Omega}\int_{\Smone}fP_{k,r}^-T_{k,r}R_{k,r}gdS(\bu)d\bx\nonumber\\
=&\int_{\Omega}\int_{\Smone}fP_{k,r}^- T_k^*R_kP_k^+gdS(\bu)d\bx-\int_{\partial\Omega}\int_{\Smone}fP_{k,r}^-d\sigma_{\bx}R_kP_k^+gdS(\bu)\nonumber\\
&+\int_{\partial\Omega}\int_{\Smone}fP_{k,r}^-T_{k,r}d\sigma_{\bx}P_k^+gdS(\bu).
\end{align}
Similarly, one can obtain
\begin{align}\label{id3}
&\int_{\Omega}\int_{\Smone}fP_{k,r}^+T_{k,r}^*Q_{k,r}gdS(\bu)d\bx\nonumber\\
=&\int_{\Omega}\int_{\Smone}fP_{k,r}^+T_kQ_kP_k^-gdS(\bu)d\bx-\int_{\partial\Omega}\int_{\Smone}fP_{k,r}^+d\sigma_{\bx}Q_kP_k^-gdS(\bu)\nonumber\\
&+\int_{\partial\Omega}\int_{\Smone}fP_{k,r}^+T_{k,r}^*d\sigma_{\bx}P_k^-gdS(\bu),
\end{align}
and
\begin{align}\label{id4}
&\int_{\Omega}\int_{\Smone}fP_{k,r}^-Q_{k,r}^2gdS(\bu)d\bx\nonumber\\
=&\int_{\Omega}\int_{\Smone}fP_{k,r}^-Q_k^2P_k^-gdS(\bu)d\bx-\int_{\partial\Omega}\int_{\Smone}fP_{k,r}^-d\sigma_{\bx}Q_kP_k^-gdS(\bu)\nonumber\\
&+\int_{\partial\Omega}\int_{\Smone}fP_{k,r}^-Q_{k,r}d\sigma_{\bx}P_k^-gdS(\bu).
\end{align}
Now, we plug \eqref{id1}-\eqref{id4} into \eqref{id0}, we obtain
\begin{align*}
&\int_{\Omega}\int_{\Smone}\Dtwo f(\bx,\bu)g(\bx,\bu)dS(\bu)d\bx\\
=&\int_{\Omega}\int_{\Smone} f(\bx,\bu)\Dtwo g(\bx,\bu)dS(\bu)d\bx\\
&+\int_{\partial\Omega}\int_{\Smone}fP_{k,r}^+d\sigma_{\bx}\bigg(-R_kP_k^++\frac{2Q_kP_k^-}{m+2k-4}\bigg)gdS(\bu)\\
&+\int_{\partial\Omega}\int_{\Smone}fP_{k,r}^-d\sigma_{\bx}\bigg(-\frac{2R_kP_k^+}{m+2k-4}+\frac{(m+2k)Q_kP_k^-}{m+2k-4}\bigg)gdS(\bu)\\
&+\int_{\partial\Omega}\int_{\Smone}f\bigg(P_{k,r}^+R_{k,r}+\frac{2P_{k,r}^-T_{k,r}}{m+2k-4}\bigg)d\sigma_{\bx}P_k^+gdS(\bu)\\
&+\int_{\partial\Omega}\int_{\Smone}f\bigg(-\frac{2P_{k,r}^+T_{k,r}^*}{m+2k-4}-\frac{(m+2k)P_{k,r}^-Q_{k,r}}{m+2k-4}\bigg)d\sigma_{\bx}P_k^-gdS(\bu),
\end{align*}
as desired. We remind the reader that $\Dtwo$ on the right hand side above is obtained from the second expression of $\Dtwo$ given in the very beginning of the proof.
\end{proof}
\begin{remark}
If we replace $g$ by the fundamental solution of $\Dtwo$ in the previous theorem, one can have the Borel-Pompeiu formula obtained in \cite{DR} with a standard argument at the singular point.
\end{remark}



\end{document}